\documentclass[12pt,oneside]{article}
\usepackage[OT4]{fontenc}
\usepackage{a4}
\usepackage{amsmath}
\usepackage{amsfonts}
\usepackage{color}
\usepackage{amssymb}

\usepackage{epic}
\usepackage{eepic}

\usepackage{graphicx}

\newtheorem{thm}{Theorem}
\newtheorem{lem}[thm]{Lemma}

\newtheorem{wn}[thm]{Corollary}
\newtheorem{df}{Definition}
\newtheorem{obs}[thm]{Observation}

\newenvironment{proof}{{\bf Proof.}}{\hspace*{\fill} \rule{2mm}{2mm} \par \hspace{0.1mm}}

\title{Total domination multisubdivision number of a graph}
\author{ D. Avella-Alaminos $^1$, M. Dettlaff $^{2}$, M. Lema\'nska $^{2}$, R. Zuazua $^{1}$
\\
\\
$^1${\small  Universidad Nacional Aut\'onoma de M\'exico, Mexico,} 
\\ {\small avella\@@matematicas.unam.mx, ritazuazua\@@ciencias.unam.com}
\\
$^2${\small Gda\'nsk University of Technology, Poland,}
\\ {\small mdettlaff\@@mif.pg.gda.pl, magda\@@mif.pg.gda.pl }
\\
}

\date{}

\begin{document}
\maketitle
\begin{abstract} 
\emph{The domination multisubdivision number} of a nonempty graph $G$ was defined in \cite{magdad} as the minimum positive integer $k$ such that there exists an edge which must be subdivided $k$ times to increase the domination number of $G$. Similarly we define the \emph{total domination multisubdivision number} msd$_{\gamma_t}(G)$ of a graph $G$ and we show that for any connected graph $G$ of order at least two, msd$_{\gamma_t}(G)\leq 3.$ We show that for trees the total domination multisubdivision number is equal to the known total domination subdivision number. We also determine the total domination multisubdivision number for some classes of graphs and characterize trees $T$ with msd$_{\gamma_t}(T)=1$. 
\end{abstract}



\section{Introduction}
In this paper we consider connected graphs with $n\geq 2$ vertices and we use $V=V(G)$ and $E=E(G)$ for the vertex set and the edge set of a graph $G.$  The \emph{neighbourhood} $N_{G}(v)$ of a vertex $v\in V(G)$ is the set of all vertices adjacent to $v$, the \emph{closed neighbourhood} $N_{G}[v]$ of a vertex $v\in V(G)$ is $N(v)\cup \{v\}$. The \emph{degree} of a vertex $v$ is $d_{G}(v)=|N_{G}(v)|.$ The \emph{distance} between two vertices $a$ and $b$, denoted by $d_G(a,b)$, is the length of the shortest $ab$-path in $G$. For a subset of vertices $X\subseteq V(G)$, the distance $d(a,X)=\min\{d(a,x):\ x\in X\}$.  The \emph{diameter} diam$(G)$ of a connected graph $G$ is the maximum distance between two vertices of $G$.

We say that a vertex $v$ of a graph $G$ is an {\it end vertex} or a leaf if $v$ has exactly one neighbour in $G$. We denote the set of all leaves in $G$ by $\Omega(G).$ A vertex $v$ is called a {\it support vertex} if it is adjacent to a leaf. If $v$ is adjacent to more than one leaf, then we call $v$ a {\it strong support vertex}. The edge incident with a leaf is called a {\it pendant edge}, in the other case we call it an {\it inner edge}.

The \emph{private neighbourhood of a vertex u with respect to a set} $D \subseteq V (G)$, where $u \in D$, is the set $PN_G[u,D] = N_G[u] - N_G[D - \{u\}]$. If $v \in PN_G[u,D]$, then we say that $v$ is a private neighbour of $u$ with respect to the set $D$.

A subset $D$ of $V(G)$ is \emph{dominating} in $G$ if every vertex of $V(G)-D$ has at least one neighbour in $D.$ Let $\gamma (G)$ be the minimum cardinality among all dominating sets in $G.$ A dominating set $D$ in $G$ with   $|D|=\gamma(G)$ is called a \emph{$\gamma (G)$-set} or a minimum dominating set of $G$ .

For a graph $G=(V,E)$, subdivision of the edge $e=uv\in E$ with vertex $x$ leads to a graph with vertex set $V\cup \{x\}$ and edge set $(E-\{uv\})\cup \{ux,xv\}$. 
Let $G_{e,t}$ denote the graph $G$ with subdivided edge $e$ with $t$ vertices (instead of edge $e=uv$ we put a path $(u,x_1,x_2,\ldots ,x_t,v)$). For $t=1$ we write $G_e.$ The vertices $\{ x_1,x_2,...,x_t\}$ are called subdivision vertices.

The \emph{domination subdivision number}, sd$_{\gamma}(G)$, of a graph $G$ is the minimum number of edges which must be subdivided (where each edge can be subdivided at most once) in order to increase the domination number. We consider subdivision number for connected graphs of order at least $3$, since the domination number of the graph $K_2$ does not increase when its only edge is subdivided. The domination subdivision number was defined in \cite{vel} and studied for example in \cite{fav1,myn,fav2}.

Let $G$ be a connected graph of order at least~2. By msd$_{\gamma}(uv)$ we denote the minimum number of subdivisions of the edge $uv$ such that $\gamma(G)$ increases. In \cite{magdad}, the \emph{domination multisubdivision number} of $G$, denoted by msd$_{\gamma }(G)$, was defined, as

$$ msd_{\gamma } (G)=\min \{msd_{\gamma }(uv):\ uv\in E(G)\}. $$

A set $S$ of vertices in a graph $G$ is a \emph{total dominating} set of $G$ if every vertex og $G$ is adjacent to a vertex in $S.$ The \emph{total domination number} $\gamma_t(G)$ is the minimum cardinality of a total dominating set of $G.$ A total dominating set $S$ in $G$ with   $|S|=\gamma_t(G)$ is called a \emph{$\gamma_t (G)$-set} or a minimum total dominating set of $G$ .
The \emph{total domination subdivision number} sd$_{\gamma_t}(G)$ of a graph $G$ (defined in \cite{haynes}) is the minimum number of edges that must be subdivided (where each edge in $G$ can be subdivided at most once) in order to increase the total domination number.

Similarly like above we define the total domination multisubdivision number of a graph~$G$.
\begin{df}
Let msd$_{\gamma_t}(uv)$ be the minimum number of subdivisions of the edge $uv$ such that $\gamma_t(G)$ increases. The \emph{total domination multisubdivision number} of a graph $G$ of order at least $2$, denoted by msd$_{\gamma_t}(G)$, is defined as
$$ msd_{\gamma_t}(G)=\min \{msd_{\gamma_t}(uv):\ uv\in E(G)\}. $$
\end{df}

For any unexplained terms see \cite{fund}. 
\section{Preliminary results}

In this section we determine the total domination multisubdivision number for some classes of graphs and we prove that for any connected graph $G$ of order at least $2$ we have msd$_{\gamma_t}(G)\leq 3$.  Let $G$ be a graph. It is clear  that sd$_{\gamma_t}(G)=1 \textrm{ if and only if msd}_{\gamma_t}(G)=1.$ 

We start with the next useful observation.
\begin{obs}
If $G$ is not a star, then it is always possible to find a $\gamma_t(G)$-set $D$ such that $D\cap \Omega(G)=\emptyset$.
\label{property}
\end{obs}

In \cite{haynes}, it has been shown that for any graph $G$ with adjacent support vertices sd$_{\gamma_t}(G)=1.$

Similarly like for the domination subdivision number in~\cite{myn} we have the next result.

\begin{lem}\label{o2} If $G$ contains an end vertex not belonging to any minimum total dominating set of $G$ or if there is an inner edge $xy$ in $G$ such that $x,y$ do not belong to any minimum total dominating set of $G,$ then sd$_{\gamma_t}(G)=1.$
\end{lem}

\begin{proof} Let $u$ be an end vertex not belonging to any $\gamma_t(G)$-set, let $v$ be the only neighbour of $u$ in $G$ and let $G'$ be a graph obtained from $G$ by a subdivision of the edge $uv$ by a vertex $w.$ By Observation~\ref{property}, exists $D'$ a minimum total dominating set with no end vertex of $G'$. Then $v,w\in D'$. The set $(D'-\{w\})\cup \{u\}$ is a  total dominating set of $G.$ Since this set contains  $u,$ it is not a minimum total dominating set of $G.$ Thus $\gamma_t(G)< |(D'-\{w\})\cup \{u\}|=|D'|$ and  sd$_{\gamma_t}(G)=1.$ 

Now suppose that there is an inner edge $xy$ in $G$ such that $x,y$ do not belong to any minimum total dominating set of $G$. Let $G'$ be a graph obtained by subdividing $xy$ with the vertex $w$ and consider any $\gamma_t(G')$-set $D'.$  If $w\notin D',$ then $D'$ is a total dominating set of $G$ containing $x$ or $y$ and by hypothesis $|D'|>\gamma_t(G)$, so we are done.

Now assume $w\in D'.$ Then $D'\cap \{ x,y\} \neq \emptyset $.  Without loss of generality suppose $x\in D'.$ Then $D=(D'-\{ w\} )\cup \{ y\}$ is a total dominating set of $G$ containing $x$ and $y$. From the assumption, it can not be minimum and similarly like before $\gamma_t(G) < |D|\leq |D'|.$ 
\end{proof}

The next lemma gives us a sufficient condition to have the total domination multisubdivision number equal to two.

\begin{lem}\label{l1} If $G$ with order $n\geq 3$ has a universal vertex, then msd$_{\gamma_t}(G)=2.$
\end{lem}
\begin{proof} If $G$ has a universal vertex $v$ then $\gamma _t (G)= 2$. If we subdivide a edge $e=vx$ by a subdivision vertex $w$, then $D=\{ v,w\} $ is a minimum total dominating set of $G_e$. If $e=yz $ with $v\notin \{ y,z\} $, then $D=\{ v,y\}$ is a minimum total dominating set of $G_e$. So, $msd_{\gamma _t}(G) > 1$. For $e=vx$, $\gamma _t(G_{e,2})=3$. Therefore,  msd$_{\gamma_t}(G)=2.$

\end{proof}

\begin{wn}
For a complete graph $K_n,$ a star $K_{1,n-1}$ with $n\geq 3$ and for a wheel $W_n$ with $n\geq 4$, we have
$$msd_{\gamma_t}(K_n)=msd_{\gamma_t}(K_{1,n-1})= msd_{\gamma_t}(W_n)=2.$$
\end{wn}

In \cite{haynes2} it has been shown that
for a cycle $C_n$ and a path $P_n$, $n\geq 3$, we have
\begin{displaymath}
sd_{\gamma_t}(C_n)=sd_{\gamma_t}(P_n)=\left\{ \begin{array}{ll}
3 & \textrm{if $n\equiv 2\pmod 4$}\\
2 & \textrm{if $n\equiv 3\pmod 4$}\\
1 & \textrm{otherwise.}
\end{array} \right.
\end{displaymath}

Since the cycle (path) with a subdivided edge $k$ times is isomorphic to the cycle (path) with subdivided $k$ edges once, we immediately obtain the following.

\begin{wn}
For a cycle $C_n$ and a path $P_n$, $n\geq 3$, we have
\begin{displaymath}
msd_{\gamma_t}(C_n)= msd_{\gamma_t}(P_n)=\left\{ \begin{array}{ll}
3 & \textrm{if $n\equiv 2\pmod 4$}\\
2 & \textrm{if $n\equiv 3\pmod 4$}\\
1 & \textrm{otherwise.}
\end{array} \right.
\end{displaymath}
\end{wn}

The main result of this section is the next theorem.

\begin{thm}\label{obs1}
For a connected graph $G$, $msd_{\gamma_t}(G)\leq 3.$
\end{thm}
\begin{proof}
We subdivide an edge $e=uv\in E(G)$ with subdivision vertices $x_1,x_2,x_3$. Let $D^*$ be a minimun total dominating set of $G_e,3$. Since $D^*$ is dominating, it contains at least one subdivision vertex. We considerer the next three cases. 
\begin{enumerate}
\item  If $|\{x_1,x_2,x_3\}\cap D^*|=1$, then $u,v\in D^*$ and $D=D^*-\{x_1,x_2,x_3\}$ is a total dominating set of $G$ with $|D|<|D^*|$.

\item Suppose $|\{x_1,x_2,x_3\}\cap D^*|=2$. If  $u\in D^*$ or $v\in D^*$, 
then $D=(D^*-\{x_1,x_2,x_3\})\cup\{u,v\}$ is a total dominating set of $G$ with $|D|<|D^*|$. If $u \not \in D^*$ and 
$v\notin D^*$, 
then the two subdivision vertices in $D^*$ must be adjacent, without loss of generality suppose $x_1,x_2\in D^*$. Then $v$ is dominated by a vertex $z\in D^*$, so $D=D^*-\{x_1,x_2\}\cup\{v\}$ is a total dominating set of $G$ with $|D|<|D^*|$.

\item If $\{x_1,x_2,x_3\}\subset D^*$, then $D=(D^*-\{x_1,x_2,x_3\})\cup\{u,v\}$ is a total dominating set of $G$ with $|D|<|D^*|$.

\end{enumerate}
In any case, we prove that $\gamma_t(G)\leq |D|< |D^*|=\gamma_t(G_{uv,3})$. Which implies that $msd_{\gamma_t}(G)\leq 3.$
\end{proof}

\begin{figure}[htp]
\begin{picture}( 220,120)(-130,0)

\drawline(0,0)(150,0)(75,130)(0,0)
\drawline(60,40)(90,40)(75,65)(60,40)
\drawline(0,0)(60,40)
\drawline(150,0)(90,40)
\drawline(75,130)(75,65)
\put(0,0){\circle*{5}}
\put(150,0){\circle*{5}}
\put(75,130){\circle*{5}}
\put(60,40){\circle*{5}}
\put(90,40){\circle*{5}}
\put(75,65){\circle*{5}}

\put(75,85){\circle*{5}}
\put(75,105){\circle*{5}}
\put(20,13){\circle*{5}}
\put(130,13){\circle*{5}}
\put(40,27){\circle*{5}}
\put(110,27){\circle*{5}}

\end{picture}\caption{Graph $G^*$}\label{gwiazdka}\end{figure}

In \cite{haynes1} it has been proved that for any positive integer $k,$ there exists a graph $G$ such that sd$_{\gamma_t}(G)=k.$ Therefore by the above Theorem, in general, the difference between sd$_{\gamma_t}(G)$ and msd$_{\gamma_t}(G)$ cannot be bounded by any integer.
For small values of sd$_{\gamma_t}$ ($2\leq \textrm{sd}_{\gamma_t}(G)\leq 3$), msd$_{\gamma_t}$ and sd$_{\gamma_t}$ are incomparable. For example, for a complete graph $K_4$ we have msd$_{\gamma_t}(K_4)=2$, sd$_{\gamma_t}(K_4)=3.$ But for the graph $G^*,$ shown in Figure \ref{gwiazdka}, we have msd$_{\gamma_t}(G^*)=3$ and sd$_{\gamma_t}(G^*)=2.$

\section{Total domination multisubdivision number of trees}
Now we consider the total domination multisubdivision number of trees. The main result of this section is the following theorem.

\begin{thm}\label{main}
For a tree $T$  with $n(T)\geq 3$ we have
\[
sd_{\gamma_t}(T)=msd_{\gamma_t}(T).
\]
\end{thm}

It was shown by Haynes \emph{et al.} in \cite{haynes} that the total domination subdivision number of a tree is 1, 2 or 3. The class of trees $T$ with sd$_{\gamma}(T)=3$ was characterized in \cite{haynes2}.

Since sd$_{\gamma_t}(G)=1$ if and only if msd$_{\gamma_t}(G)=1$, in order to prove Theorem ~\ref{main} it suffices to show that for any tree $T$ of order at least three,
\begin{equation*}\label{e1}
sd_{\gamma_t}(T)=3 \textrm{ if and only if } msd_{\gamma_t}(T)=3.
\end{equation*}

\subsection{Trees with the total domination multisubdivision number equal to 3}
The following constructive characterization of the family $\mathcal{F}$ of labeled trees $T$ with sd$_{\gamma_t}(T)=3$ was given in~\cite{haynes2}. The label of a vertex $v$ is also called the status of $v$ and is denoted by $sta(v)$.

 Let $\mathcal{F}$ be the family of labelled trees such that:
 \begin{itemize}
\item contains $P_6$ where the two leaves have status $C,$ the two support vertices have status $B,$ and the two central vertices have
status $A;$ and
\item is closed under the two operations $O_1$ and $O_2$, which extend the tree $T$ by attaching a path to a vertex $y \in V (T )$.
\begin{enumerate}
\item Operation $O_1.$ Assume $sta(y) = A.$ Then add a path $(x,w,v)$ and the edge $xy.$ Let $sta(x) = A,$ $sta(w) = B,$ and $sta(v) = C.$
\item Operation $O_2.$ Assume $sta(y) \in \{B,C\}.$ Then add a path $(x,w,v,u)$ and the edge $xy.$ Let $sta(x)=sta(w)=A, sta(v)=B$ and $sta(u)=C.$
\end{enumerate}
\end{itemize}

In \cite{haynes2} the following observation and theorem has been proved.

\begin{obs} \label{orodzina} If $T\in \mathcal{F},$ then $\mathcal{B}\cup \mathcal{C}$ is a minimum total dominating set of $T,$ where $\mathcal{B}$ and $\mathcal{C}$ are sets of vertices with status $B$ and $C,$ respectively. 
\end{obs}

\begin{thm}  \label{rodziny}
 For a tree $T$, sd$_{\gamma_t}(T)=3$ if and only if $T \in \mathcal{F}.$
\end{thm}

Operation $O_1$ and Operation $O_2$ will be called below the basic operations. If $S$ is a basic operation of type $O_1$ or $O_2$, denote by $V_S$ and $E_S$ the set of vertices and the set of edges appeared as a result of using the operation $S$.

\begin{obs}\label{commute}
Let $T \in \mathcal{F}$ and $S$, $S'$ be two basic operations. Consider $S'(S(T))$, if the path added by $S'$ is attached to a vertex  $v\in V(T)$, then $S'(S(T))=S(S'(T))$.
\end{obs}

\begin{lem}\label{change}
Let $T \in \mathcal{F}$ with $|V(T)|>6$. Then there exist  $T',T'' \in \mathcal{F}$ and basic operations $S'$, $S''$ such that $T=S'(T')=S''(T'')$ and $V_{S'}\cap V_{S''}=\emptyset$. Additionally,  $E_{S'}\cap E_{S''}=\emptyset$.
\end{lem}

\begin{proof}
We use induction on $n,$ the number of vertices of $T$. Any  $T \in \mathcal{F}$ with $n>6$ has at least $9$ or $10$ vertices. For $n=9$, $T=S'(T')$ for $T'$ the path $(v_1,v_2,v_3,v_4,v_5,v_6)$ and $S'$ the operation of type $O_1$ adding path $(x,w,v)$   attached to vertex $v_3$; then $T=S''(T'')$ for $T''$ the path $(v_1,v_2,v_3,x,w,v)$ and $S''$ the operation of type $O_1$ adding path $(v_4,v_5,v_6)$  attached to vertex $v_3$. Obviously $V_{S'}\cap V_{S''}=\emptyset$.
For $n=10$ we have two cases, $T=S'(T')$ for $T'$ the path $(v_1,v_2,v_3,v_4,v_5,v_6)$ and $S'$ the operation of type $O_2$ adding path $(x,w,v,u)$ attached to vertex $v_5$; then $T=S''(T'')$ for $T''$ the path $(v_6,v_5,x,w,v,u)$ and $S''$ the operation of type $O_2$ adding path $(v_4,v_3,v_2,v_1)$  attached to vertex $v_5$. The second case is $T=S'(T')$ for $T'$ the path $(v_1,v_2,v_3,v_4,v_5,v_6)$ and $S'$ the operation of type $O_2$ adding path $(x,w,v,u)$   attached to vertex $v_6$; then $T=S''(T'')$ for $T''$ the path $(v_5,v_6,x,w,v,u)$ and $S''$ the operation of type $O_2$ adding path $(v_4,v_3,v_2,v_1)$  attached to vertex $v_5$. In both cases $V_{S'}\cap V_{S''}=\emptyset$.

Let $T \in \mathcal{F}$ with $n>10$, and suppose the result holds for every tree of $\mathcal{F}$ with less than $n$ vertices. By definition of the family $\mathcal{F}$ we know $T=S(\hat{T})$, for some $\hat{T}\in\mathcal{F}$ and a basic operation $S$. By induction hypothesis,  there exist  $T',T'' \in \mathcal{F}$ and basic operations $S'$, $S''$ such that
$\hat{T}=S'(T')=S''(T'')$,  $V_{S'}\cap V_{S''}=\emptyset$, and then $T=S(S'(T'))=S(S''(T''))$. The path added by $S$ is attached to a vertex $v\in \hat{T}$, and since   $V_{S'}\cap V_{S''}=\emptyset$, $v$ does not belong to both  $V_{S'}$ and $V_{S''}$, without loss of generality, $v\notin V_{S''}$, so by Observation \ref{commute}, $S(S''(T''))=S''(S(T''))$. Then $T=S(S'(T'))=S''(S(T''))$, with $V_{S}\cap V_{S''}=\emptyset$.
\end{proof}

With the above result we can prove the next lemma.

\begin{lem}
If $T$ is a tree with sd$_{\gamma_t}(T)=3$, then msd$_{\gamma_t}(T)=3$.
\label{t1}
\end{lem}
\begin{proof} From Theorem \ref{rodziny}, it is enough to prove that if $T\in \mathcal{F},$ then msd$_{\gamma_t}(T)=3$. We prove that for any edge $e$ of $T\in \mathcal{F},$ $\gamma_t(T_{e,2})=\gamma_t(T).$ We use induction on $n,$ the number of vertices of $T.$

By Corollary $6$,  the result is true for a path $P_6.$ Assume that for every tree $T'$ with $n'< n$ vertices belonging to the family ~$\mathcal{F},$ equality $\gamma_t(T'_{e,2})=\gamma_t(T')$ holds for any edge $e$ of $T'$. 

Let $T\in \mathcal{F}$ be a tree with $n>6$ vertices and let $e$ be any edge of $T.$ Since $T\in \mathcal{F},$ $T=T_j$ and is constructed from $P_6$ by applying $j-1$ basic operations. By Lemma~\ref{change} we can assume that $e\in E(T_{j-1})$. Since $|V(T_{j-1})|<|V(T_{j})|$, from the induction hypothesis, $\gamma_{t}((T_{j-1})_{e,2})=\gamma_t(T_{j-1}).$ Using Observation \ref{orodzina} we know $\gamma_t(T)=\gamma_t(T_{j-1})+2.$

We consider two cases:

$Case$ $1.$ If $T=T_j=O_1(T_{j-1})$ then we added a path $(x,w,v)$ to a vertex  of $T_{j-1}$ with status $A$.   If $D'$ is a minimum total dominating set of $(T_{j-1})_{e,2},$ then $D_1=D'\cup \{v,w\}$ is a total dominating set of $T_{e,2}$ with $|D_1|=\gamma_t(T_{j-1})+2=\gamma_t(T)$, so $\gamma_t(T_{e,2})\leq \gamma_t(T)$. Then  $\gamma_t(T_{e,2})= \gamma_t(T)$.

$Case$ $2.$ If $T=T_j=O_2(T_{j-1})$ then we added a path $(x,w,v,u)$ to a vertex  of $T_{j-1}$ with status $B$ or $C$. If $D'$ is a minimum total dominating set of $(T_{j-1})_{e,2},$ then $D_1=D'\cup \{w,v\}$ is a total dominating set of $T_{e,2}$ with $|D_1|=\gamma_t(T_{j-1})+2=\gamma_t(T)$, so $\gamma_t(T_{e,2})\leq  \gamma_t(T)$. Then  $\gamma_t(T_{e,2})=\gamma_t(T)$.
\end{proof}

The next observation and lemmas are necessaries in order to finish the proof of Theorem~\ref{main}.

\begin{obs}\label{strong}
If T is a tree with msd$_{\gamma_t}(T)=3$, then $T$ does not have a strong support vertex.
\end{obs}
\begin{proof}
Suppose msd$_{\gamma_t}(T)=3$ and $T$ has a strong support vertex $v$ adjacent to a leaf $u$. Let us subdivide the edge $e=uv$ with two vertices $a,b$ and let $D'$ be a minimum total dominating set with no end vertex of $T_{e,2}$. It is clear that  $a,b\in D'$. Since $v$ is a support in $T_{e,2}$, $v\in D'$. Hence, $(D'-\{a,b\})\cup \{u\}$ is a total dominating set in $T$, what implies $\gamma_t(T)\leq |D'|-1<\gamma_t(T_{e,2})$, a contradiction with msd$_{\gamma_t}(T)=3$. 
\end{proof}

\begin{lem}
Let $T$ be a tree with $n>6$ vertices such that msd$_{\gamma_t}(T)=3.$ Let $P=(v_0,\ldots,v_l)$ be a longest path of $T$ $(l\geq 5)$ and let $D$ be a minimum total dominating set with no end vertex of $T$. Then:

\begin{enumerate}
\item $d_T(v_1)=d_T(v_2)=2$;
\item  $v_3$ is not a support vertex. Moreover, if  $d_T(v_3)>2$, outside the path $P$, only  one $P_2$ path or $P_3$ paths  may be attached to $v_3$ and for $T'=T-\{ v_0,v_1,v_2\}$,   $\gamma_t(T)=\gamma_t(T')+2$. 
\end{enumerate}
\label{riti}
\end{lem}

\begin{proof} Let $D$ be a minimum total dominating set with no end vertex of $T$.
\begin{enumerate}

\item It is clear that $v_1,v_2\in D.$ By Observation~\ref{strong}, $d_T(v_1)=2.$ 
Suppose $d_T(v_2)>2.$ For the edge $e=v_0v_1$ consider the tree $T_{e,2},$ where we subdivide $e$ by two vertices $a,b.$ If $D'$ is a minimum total dominating set with no end vertex of $T_{e,2}$, then $a,b\in D'.$ If $v_2$ is a support vertex, then $v_2\in D'.$ If $v_2$ is not a support vertex, then it is a neighbour of a support vertex of degree two and in this case also $v_2\in D'.$  Then $(D'-\{a,b\})\cup \{v_1\}$ is a total dominating set of $T,$ a contradiction with msd$_{\gamma_t}(T)=3$. Thus $d_T(v_2)=2.$

\item Suppose  $v_3$ is a support vertex adjacent to a leaf $y$.  Consider $T_{e,2},$ where $e=v_3y$ and denote the two vertices on the subdivided edge  by $a,b.$ If $D'$ is a minimum total dominating set with no end vertex of $T_{e,2}$, then $a,b,v_1,v_2\in D'$. Then $(D'-\{a,b\})\cup \{v_3\}$ is a total dominating set of $T,$ a contradiction with msd$_{\gamma_t}(T)=3$.

Suppose  $d_T(v_3)>2$. If $d_{T}(v_3,\Omega(T))=2$, then $v_3$ is adjacent to a support vertex $x$ which is a neighbour of a leaf $y$. By Observation~\ref{strong} $x$ is not a strong support vertex, if  $d_T(x)>2$  then $x$  belongs to a longest path of $T$ and by $1$, $d_T(x)=2$, a contradiction. Since msd$_{\gamma_t}(T)=3$  outside the path $P$, only one $P_2$ path may be attached to $v_3$. Now, if $d_{T}(v_3,\Omega(T))=3$, then there are vertices $x,y,z$ such that $(z,y,x,v_3,\ldots,v_l)$  is a longest path of $T$ and by $1$, $d_T(x)=d_T(y)=2$. Hence, outside the path $P$, only  $P_3$'s may be attached to $v_3$.

Observe that for any minimum total dominating set with no end vertex $D$ of $T$, $D-\{v_1,v_2\}$ is a total dominating set of $T'$.  Similarly, for any minimum total dominating set with no end vertex $D'$ of $T'$, $D'\cup \{v_1,v_2\}$ is a total dominating set of $T$ and $\gamma_t(T)\leq \gamma_t(T')+2$. Therefore, $\gamma_t(T)=\gamma_t(T')+2$.

\end{enumerate}

\end{proof}

As a consequence of the last case, if $d_T(v_3)>2$, then we can observe that every minimum total dominating set with no end vertex $D$ of $T$ has the form $D= D'\cup \{v_1,v_2\}$, where  $D'$ is a minimum total dominating set with no end vertex of $T'$. Equivalently, every $D'$ has the form $D'=D-\{v_1,v_2\}$.

\begin{lem}
If $T$ is a tree with msd$_{\gamma_t}(T)=3$, then sd$_{\gamma_t}(T)=3$.
\label{t2}
\end{lem}

\begin{proof} From Theorem \ref{rodziny}, it is enough to prove that if  $T$ is a tree with msd$_{\gamma_t}(T)=3$, then $T$ belongs to the family $\mathcal{F}.$ We use induction on $n,$ the number of vertices of a tree $T.$ The smallest tree $T$ such that msd$_{\gamma_t}(T)=3$ is a path $P_6$ and $P_6\in \mathcal{F}.$ Assume that every tree $T'$ with less than $n$ vertices such that msd$_{\gamma_t}(T')=3$ belongs to the family $\mathcal{F}.$

Let $T$ be a tree with msd$_{\gamma_t}(T)=3$ and $n>6$ vertices. Consider $P=(v_0,\ldots,v_l)$ a longest path of $T$, $l\geq 5$, and let $D$ be a minimum total dominating set with no end vertex of $T$. 

By Lemma~\ref{riti}, $d_T(v_1)=d_T(v_2)=2$. So we consider the next two cases.
\begin{enumerate}
\item $d_T(v_3)>2.$ By Lemma~\ref{riti}, $v_3$ is not a support vertex. We have the following subcases.

\begin{itemize}

\item $d_{T}(v_3,\Omega(T))=2$.  By Lemma ~\ref{riti},  outside the path $P$, only one $P_2$ path may be attached to $v_3$. Let us denote $x,y$ the vertices of that path, where $y$ is a leaf of $T$.  Again by Lemma ~\ref{riti}, for $T'=T-\{v_0,v_1,v_2\}$, $\gamma_t(T')=\gamma_t(T)-2$. 

For any $e\in E(T')-\{xy,xv_3\}$, $\gamma_t(T'_{e,2})$ $=\gamma_t(T_{e,2})-2=\gamma_t(T)-2=\gamma_t(T').$ In order to see that also for $e\in \{xy,xv_3\}$, $\gamma_t(T'_{e,2})=\gamma_t(T')$, we claim that  there exists a $\gamma_t(T')$-set $D^*$ with no end vertex such that $v_4\in D^*$ and $|N_{T'}(v_4)\cap D^*|\geq 2$. 

Proof of the claim: Consider $T_{e,2},$ where $e=v_3v_4$, and denote the two sudivision vertices by $a,b.$ If $D'$ is a minimum total dominating set with no end vertex of $T_{e,2},$ then $\{v_1,v_2,x,v_3\}\subset D'$. If  $\{ a,b\} \cap D'\neq \emptyset$, then $D=D'-\{ a,b\}$ is a total dominating set of $T$  with $|D|<\gamma_t(T_{e,2})$, which is a contradiction with $\gamma_t(T)=\gamma_t(T_{e,2})$. Therefore, there exists $z\in N_{T_{e,2}}(v_4)$, $z\neq b$ such that $\{ v_4, z\} \subset D'$, and then $D^*=D'-\{ v_1,v_2\}$  is a $\gamma_t(T')$-set with no end vertex such that $v_4\in D^*$ and $|N_{T'}(v_4)\cap D^*|\geq 2$. 

Now, without loss of generality, consider $e=xy$ and subdivision of the edge $xy$ with vertices $c,d$. We know that  $(D^*-\{x,v_3\})\cup \{c,d\}$ is a total dominating set in $T'_{xy,2}$, so $\gamma_t(T'_{e,2})=\gamma_t(T')$.

Finally, for any edge $e\in E(T')$ we have $\gamma_t(T')=\gamma_t(T'_{e,2})$. Thus msd$_{\gamma_t}(T')=3$ and from the induction hypothesis $T'\in \mathcal{F}.$ Since $sta(v_3)=A,$ it is possible to obtain $T$ from $T'$ by Operation $O_1.$ It implies that $T\in \mathcal{F}.$

\item $d_{T}(v_3,\Omega(T))=3.$ Thus,   by Lemma  ~\ref{riti},  outside the path $P$, only $P_3$'s may be attached to $v_3$. Let us denote $x,y,z$ the vertices of one of such paths, where $z$ is a leaf of $T$. Define $T'=T-\{v_0,v_1,v_2\}$.

For any $e\in E(T')-\{xy,yz,xv_3\}$, $\gamma_t(T'_{e,2})=\gamma_t(T_{e,2})-2=\gamma_t(T)-2=\gamma_t(T').$ Since msd$_{\gamma_t}(T)=3$ and by Lemma ~\ref{riti}, $\gamma_t(T')=\gamma_t(T)-2$, there exists a $\gamma_t(T')$-set $D^*$ with no end vertex such that $\{x, y, v_3, v_4\} \subset D^*$ (if not, then $\gamma_t(T_{v_3v_4,2})>\gamma_t(T)$, a contradiction).
It is enough to consider subdivision of the edge $yz$ with vertices $a,b$. Hence $(D^*-\{x,y\})\cup \{a,b\}$ is a total dominating set in $T'_{yz,2}$. Finally, for any edge $e\in E(T')$ we have $\gamma_t(T')=\gamma_t(T'_{e,2})$. Thus msd$_{\gamma_t}(T')=3$ and from the induction hypothesis $T'\in \mathcal{F}.$ Since $sta(v_3)=A$, it is possible to obtain $T$ from $T'$ by Operation $O_1.$ Hence, $T\in \mathcal{F}.$
\end{itemize}
\item $d_T(v_3)=2.$ We have two subcases.
\begin{itemize}
 \item $d_T(v_4)=2$ or ($d_T(v_4)>2$ and $d_T(v_4,\Omega(T))\in \{1,4\}$). It is clear that $v_1,v_2\in D$ for any minimum total dominting set without end vertex of $T$. Without lost of generality we can suppose that $v_3\notin D$.  If we consider $T'=T-\{v_0,v_1,v_2,v_3\},$ then  $\gamma_t(T')=\gamma_t(T)-2$ and for any $e\in E(T')$,  $\gamma_t(T'_{e,2})=\gamma_t(T_{e,2})-2=\gamma_t(T)-2=\gamma_t(T').$ Thus msd$_{\gamma_t}(T')=3$, from the induction hypothesis $T'\in \mathcal{F}$ and by the definition of the family $\mathcal{F}$, the status of the vertex $v_4$ is $B$ or $C$.  So $T$ can be obtained from $T'$ by Operation $O_2,$ what implies $T\in \mathcal{F}.$

\item $d_T(v_4,\Omega(T))\in \{2,3\}.$ Suppose $d_T(v_4,\Omega(T))=2$ , then $v_4$ is adjacent to a support vertex $y$. Consider $T_{e,2},$ where $e=v_3v_4$ and denote the two subdivision vertices by $a,b.$ If $D'$ is a minimum total dominating set with no end vertex of $T_{e,2},$ then $v_1,v_2,y,v_4\in D'$. Since $D'$ is total dominating then 
there exist $z\in D\cap \{b,v_3\}\neq \emptyset$ such that  $D'-\{z\}$ is a total dominating set of $T,$ a contradiction with msd$_{\gamma_t}(T)=3$.
The case of $d_T(v_4,\Omega(T))=3$ is similar. 
\end{itemize}
\end{enumerate}

\end{proof}

\subsection{Trees with the total domination multisubdivision number equal to 1}
In this section we give a characterization of trees $T$ of order at least three with sd$_{\gamma_t}(T)=\textrm{msd}_{\gamma_t}(T)=1$. In order to prove the main Theorem~\ref{final} we need the next technical lemmas.

\begin{lem}
Let $T$ be a tree of order $n\geq 3$ such that
\begin{enumerate}
\item for any end-vertex $u$ there exists a $\gamma_t(T)$-set $D$ such that $u\in D$ and
\item for any inner edge $uv$ there is a $\gamma_t(T)$-set $D$ such that
\begin{itemize}
\item [a)]$|\{u,v\}\cap D|=1$, say $u\in D$, and $v\not \in\ $PN$_T[u,D]$ or
\item [b)]$|\{u,v\}\cap D|=2$ and at least one of the following conditions holds:
\begin{itemize}
\item [b1)]$ |N_T(u)\cap D|\geq 2$ and $|N_T(v)\cap D|\geq 2$;  
\item [b2)]$N_T(u)\cap D=\{v\}$ and $\Big(PN_T[u,D] = \emptyset$ or $\big( PN_T[v,D] = \emptyset$ and  $|N_T(x)\cap D|\geq 2$ for any vertex $x\in (N_T(v)\cap D)-\{u\}\big) \Big)$; 
\item [b3)]$N_T(v)\cap D=\{u\}$ and $\Big(PN_T[v,D] = \emptyset$ or $\big( PN_T[u,D] = \emptyset$ and $|N_T(x)\cap D|\geq 2 $ for any vertex $x\in (N_T(u)\cap D)-\{v\}\big)\Big)$.
\end{itemize}
\end{itemize}
\end{enumerate}
Then sd$_{\gamma_t}(T)>1$.
\label{lemaRitaMagda}
\end{lem}
\begin{proof}
Let $e=uv$ be an edge of the tree $T$. Let us subdivide the edge $e$ with a vertex $w$. If $u\in \Omega (T)$, then there is a $\gamma_t(T)$-set  $D$ containing $u$ and $v$. Thus $(D-\{u\})\cup \{w\}$ is a $\gamma_t(T_{uv})$-set and $\gamma_t(T)=\gamma_t(T_{uv}).$ 

Suppose that $\{u,v\}\cap \Omega (T)=\emptyset.$ 

If $a)$ holds, then $D$ is also a $\gamma_t(T_{uv})$-set and again we obtain $\gamma_t(T)=\gamma_t(T_{uv}).$

Assume now $b)$ holds,

\begin{itemize}
\item if condition $b1)$ holds, then $D$ is also a $\gamma_t(T_{uv})$-set.
\item if condition $b2)$ holds, we have two cases: if $N_T(u)\cap D=\{v\}$ and $PN_T[u,D] = \emptyset$, then $(D-\{u\})\cup \{w\}$ is a $\gamma_t(T_{uv})$-set. If $N_T(u)\cap D=\{v\}$ and $PN_T[v,D] = \emptyset$ and for any vertex $x\in (N_T(v)\cap D)-\{u\}$ we have $|N_T(x)\cap D|\geq 2$, then $(D-\{v\})\cup \{w\}$ is a $\gamma_t(T_{uv})$-set.
\item similarly if condition $b3)$ holds.
\end{itemize}

In all the cases we have found a $\gamma_t(T_{uv})$-set of cardinality $\gamma_t(T)$. This implies that sd$_{\gamma_t}(T)> 1$.
\end{proof}

\begin{lem}
Let T be a tree of order $n\geq 3$ having an inner edge $uv\in E(T)$ such that for any $\gamma_t(T)$-set $D$ we have:
\begin{enumerate}
\item if $|\{u,v\}\cap D|=1$, let us say $u\in D$, then $v\in PN_T[u,D]$ and \item if $|\{u,v\}\cap D|=2$, then $N_T(u)\cap D=\{v\}$ or $N_T(v)\cap D=\{u\}$, and if $N_T(u)\cap D=\{v\}$, then $PN_T[u,D]\not = \emptyset$ and $\big( PN_T[v,D]\not = \emptyset$ or $N_T(x)\cap D=\{v\}$ for a vertex $x\in (N_T(v)\cap D)-\{u\} \big)$. Similarly if $N_T(v)\cap D=\{u\}$. 
\end{enumerate} 
Then sd$_{\gamma_t}(T)=1$.
\label{lemaMagdaRita}
\end{lem}
\begin{proof}
We subdivide the edge $uv$ with vertex $w$ and let $D'$ be a $\gamma_t$-set of $T_{uv}$. 

\begin{enumerate}
\item If $w\in D'$, then at least one of $u,\ v$ belongs to $D'$.
\begin{itemize}
\item Suppose $\{u,w,v\}\subseteq D'$, then $D'-\{w\}$ is a total dominating set of $T$ and $\gamma_t(T)<\gamma_t(T_{uv})$.

\item $|\{u,v\}\cap D'|=1$ and without loss of generality suppose  $\{u,w\}\subseteq D'$. Thus, if $|N_{T_{uv}}(u)\cap D'|\geq 2$, then $D'-\{w\}$ is a total dominating set of $T$ and $\gamma_t(T)<\gamma_t(T_{uv})$. In the other case, if $N_{T_{uv}}(u)\cap D'=\{w\}$, then $D=(D'-\{w\})\cup \{v\}$ is a total dominating set of $T$ such that PN$_T[v,D]=\emptyset$ and for any vertex $x\in (N_T(v)\cap D)-\{u\}$ we have $|N_T(x)\cap D|\geq 2$, so by hypothesis $2$, $\gamma_t(T_{uv})=|D|>\gamma_t(T)$.
\end{itemize}
\item If $w\not \in D'$, then we have two possibilities: 

\begin{itemize}
\item $|\{u,v\}\cap D'|=1$ and we assume, without loss of generality, $u\in D'$. Then $D'$ is a total dominating set in $T$ such that $v\not \in\ $PN$_T[u,D']$ and by hypothesis $1$, $|D'|>\gamma_t(T)$.

\item If $\{u,v\}\subseteq D'$, then $D'$ is total dominating set of $T$ such that $|N_T(u)\cap D'|\geq 2$ and $|N_T(v)\cap D'|\geq 2$, again we have that  $|D'|>\gamma_t(T)$. In all of the cases we obtained $\gamma_t(T_{uv})>\gamma_t(T)$, what implies sd$_{\gamma_t}(T)=1$.
\end{itemize}
\end{enumerate}
\end{proof}

It is straightforward that from Lemma~\ref{o2}, Lemma~\ref{lemaRitaMagda} and Lemma~\ref{lemaMagdaRita} we have the next Theorem.
\begin{thm}\label{final}
Let T be a tree of order $n\geq 3$. Then sd$_{\gamma_t}(T)=1$ if and only if $T$ has 
\begin{itemize}
\item a leaf  which does not belong to any $\gamma_t(T)$-set or
\item an inner edge $uv\in E(T)$ such that for any $\gamma_t(T)$-set $D$ 
\begin{enumerate}
\item if $|\{u,v\}\cap D|=1$, let us say $u\in D$, then $v\in PN_T[u,D]$ and
\item if $|\{u,v\}\cap D|=2$, then $N_T(u)\cap D=\{v\}$ or $N_T(v)\cap D=\{u\}$, and if $N_T(u)\cap D=\{v\}$, then $PN_T[u,D]\not = \emptyset$ and $\big( PN_T[v,D]\not = \emptyset$ or $N_T(x)\cap D=\{v\}$ for a vertex $x\in (N_T(v)\cap D)-\{u\}\big)$. Similarly if $N_T(v)\cap D=\{u\}$. 
\end{enumerate} 
\end{itemize}
\end{thm}

\textbf{Acknowledgements}

The authors thank the financial support received from Grant UNAM-PAPIIT IN-117812 and SEP-CONACyT.

\end{document}